\documentclass{TPmod}

\newcommand{\alg}{\mathsf{a}}
\newcommand{\Alg}{\mathsf{A}}
\newcommand{\algd}{\text{\usefont{U}{bbm}{m}{n}{a}}}
\newcommand{\Algd}{\text{\usefont{U}{bbm}{m}{n}{A}}}
\newcommand{\At}{\mathsf{At}}

\newcommand{\HKR}{\mathsf{HKR}}

\newcommand{\KS}{\mathsf{KS}}
\def\deform{\mathsf{def}}
\def\cB{\mathcal{B}}
\def\QH{\mathsf{QH}}
\newcommand{\kalg}{\BbK\mbox{-}\mathtt{alg}}
\newcommand{\modd}{\mbox{-}\mathtt{mod}}
\newcommand{\bimod}{\mbox{-}\mathtt{bimod}}
\def\nc{\mathsf{nc}}
\newcommand{\dbdg}[1]{D^b_{\mathsf{dg}}Coh(#1)}
\def\cM{\mathcal{M}}

\makeatletter
\def\@xfootnote[#1]{%
  \protected@xdef\@thefnmark{#1}%
  \@footnotemark\@footnotetext}
\makeatother

\begin{document}

\title{Automatic split-generation for the Fukaya category}

\author[Perutz and Sheridan]{Timothy Perutz\footnote{Partially supported by NSF grants DMS-1406418 and CAREER 1455265.} and Nick Sheridan\footnote{Partially supported by the National Science Foundation through Grant number
DMS-1310604, and under agreement number DMS-1128155. 
Any opinions, findings and conclusions or recommendations expressed in this material are those of the authors and do not necessarily reflect the views of the National Science Foundation.}}

\address{Timothy Perutz, University of Texas at Austin, Department of Mathematics, RLM 8.100, 2515 Speedway Stop C1200, Austin, TX 78712, USA.}
\address{Nick Sheridan, Department of Mathematics, Princeton University, Fine Hall, Washington Road, Princeton, NJ 08544-1000, USA.}

\begin{abstract}
{\sc Abstract:} We prove a structural result in mirror symmetry for projective Calabi--Yau (CY) manifolds. Let $X$ be a connected symplectic CY manifold, whose Fukaya category $\EuF(X)$ is defined over some suitable Novikov field $\BbK$; its mirror is assumed to be some smooth projective scheme $Y$ over $\BbK$ with `maximally unipotent monodromy'. Suppose that some split-generating subcategory of (a $\mathsf{dg}$ enhancement of) $D^bCoh( Y)$ embeds into $\EuF(X)$: we call this hypothesis `core homological mirror symmetry'. 
We prove that the embedding extends to an equivalence of categories, $D^bCoh(Y) \cong D^\pi( \EuF(X))$, using Abouzaid's split-generation criterion.  
Our results are not sensitive to the details of how the Fukaya category is set up. 
In work-in-preparation \cite{Perutz2015a}, we establish the necessary foundational tools in the setting of the `relative Fukaya category', which is defined using classical transversality theory.
\end{abstract}

\maketitle

\tableofcontents

\section{Introduction}

\subsection{Standing assumptions}
\label{subsec:setup}

Let $\Bbbk$ be a field of characteristic zero and $\BbK := \Lambda^R$ a Novikov field over $\Bbbk$, where $R \subset \R$ is an additive subgroup: that is,
\[ \BbK := \left\{ \sum_{j=0}^\infty c_j \cdot q^{\lambda_j}: c_j \in \Bbbk, \lambda_j \in R, \lim_{j \to \infty} \lambda_j = +\infty\right\}.\]

Let $(X,\omega)$ be a compact, connected, Calabi-Yau symplectic manifold of dimension $2n$ (`Calabi-Yau' here means $c_1(TX) = 0$). 

Let $Y \to \cM = \spec \BbK$ be a smooth, projective, Calabi-Yau algebraic scheme of relative dimension $n$ (`Calabi-Yau' here means the canonical sheaf is trivial).

\subsection{The Fukaya category}
\label{subsec:fukint}

We consider the Fukaya category of $X$, denoted $\EuF(X)$. 
Of course, there are a number of possibilities in defining the Fukaya category, which depend on various additional choices. 
We will always restrict ourselves to definitions where $\EuF(X)$ is a $\Z$-graded, $\BbK$-linear $A_\infty$ category (in particular, the curvature $\mu^0$ vanishes). 

In \S \ref{sec:fuk}, we give a list of properties that we need the Fukaya category $\EuF(X)$ to have in order for our results to work. 
We expect these properties to hold very generally, so we do not tie ourselves to a particular version of the Fukaya category. 

It will be proven in \cite{Perutz2015a} (in preparation) that a version called the \emph{relative Fukaya category} has all of these properties, so the range of applicability of our results is not empty. 
Let us briefly outline what that construction looks like, so the reader can keep a concrete example in mind. 

It depends on a choice of \emph{Calabi-Yau relative K\"{a}hler manifold}: that is, a Calabi-Yau K\"{a}hler manifold $(X,\omega)$, together with an ample simple normal crossings divisor $D \subset X$, and a proper K\"{a}hler potential $h$ for $\omega$ on $X \setminus D$: in particular, $\omega = d\alpha$ is exact on $X \setminus D$, where $\alpha := d^c h$. 
Its objects are closed, exact Lagrangian branes $L \subset X \setminus D$. 
Floer-theoretic operations are defined by counting pseudoholomorphic curves $u\co \Sigma \to X$, with boundary on Lagrangians in $X \setminus D$ (transversality of the moduli spaces is achieved using the stabilizing divisor method of Cieliebak and Mohnke \cite{Cieliebak2007}). 
These counts of curves $u$ are weighted by $q^{\omega(u) - \alpha(\partial u)}$: so the category is defined over $\BbK = \Lambda^{R}$, where $R$ contains the image of the map
\begin{eqnarray*}
H_2(X,X \setminus D) & \to & \R, \\
u & \mapsto & \omega(u) - \alpha(\partial u).
\end{eqnarray*}
The resulting curved $A_\infty$ category is denoted $\EuF(X,D)_{curv}$. 
We then define $\EuF(X,D)$, which is an honest $A_\infty$ category (one without curvature): its objects are objects of $\EuF(X,D)_{curv}$, equipped with bounding cochains.

The analogues of the necessary properties in the monotone case have also been established in \cite{Sheridan2013} (but in that situation, $X$ is not Calabi-Yau so our results do not apply).
It is expected that work in preparation of Abouzaid, Fukaya, Oh, Ohta and Ono \cite{Abouzaid2012} will prove that the Fukaya category of an arbitrary symplectic manifold has all of the necessary properties except for that described in \S \ref{subsec:defmap} (which ought to hold in full generality, but is easier to prove for the relative Fukaya category).

We emphasise that, if you want to apply our results to your favourite version of the Fukaya category, you just need to verify that it has the properties outlined in \S \ref{sec:fuk}.

\subsection{Split-generating the Fukaya category}

Now let us recall Abouzaid's split-generation criterion \cite{Abouzaid2010a}, adapted to the present setting (following \cite{Abouzaid2012}, see also \cite{Ritter2012} and \cite{Sheridan2013} for the monotone case). 
It concerns the open-closed string map:
\begin{equation}
\label{eqn:oc}
 \EuO\EuC \co \HH_\bullet(\EuF(X)) \to \QH^{\bullet+n}(X).
\end{equation}

\begin{thm}
\label{thm:abouzaid} 
Let $\EuA$ be a full subcategory of $\EuF(X)$. 
If the identity $e \in \QH^0(X)$ lies in the image of the map
\begin{equation}
\label{eqn:oc0}
\EuO\EuC|_\EuA \co \HH_{-n}(\EuA) \to \QH^0(X),
\end{equation}
then $\EuA$ split-generates $\EuF$.
\end{thm}

\begin{rmk}
When $X$ is a Liouville manifold, Theorem \ref{thm:abouzaid} was proved (for the wrapped Fukaya category) in \cite{Abouzaid2010a}. 
It is expected to be proven in full generality in \cite{Abouzaid2012}.
\end{rmk}

There is also a dual version, involving the closed-open string map:
\[ \EuC\EuO \co \QH^\bullet(X) \to \HH^\bullet(\EuF(X)).\]

\begin{thm}
\label{thm:coinj} 
Let $\EuA$ be a full subcategory of $\EuF(X)$. 
If the map
\[
 \EuC\EuO|_\EuA \co \QH^{2n}(X) \to \HH^{2n}(\EuA)
\]
is injective, then $\EuA$ split-generates $\EuF$.
\end{thm}

\begin{rmk}
Theorem \ref{thm:coinj} is expected to be proven for Fukaya categories of compact symplectic manifolds in \cite{Abouzaid2012}.
\end{rmk}

In \S \ref{subsec:gen}, we explain how Theorems \ref{thm:abouzaid} and \ref{thm:coinj} are proved, in particular, which of the properties from \S \ref{sec:fuk} they rely on.

\subsection{Homological mirror symmetry}

For any $A_\infty$ category $\EuC$, we denote by `$\twsplit \EuC$' the split-closed triangulated envelope (denoted `$\Pi(Tw \,\EuC)$' in \cite[\S 4c]{Seidel2008}).

Let $\dbdg{Y}$ be a $\mathsf{dg}$ enhancement of the bounded derived category of  coherent sheaves $D^bCoh(Y)$: we regard it as a $\Z$-graded, $\BbK$-linear, triangulated $A_\infty$ category. 
Because $Y$ is projective, the $\mathsf{dg}$ enhancement is unique up to quasi-equivalence, by \cite[Theorem 8.13]{Lunts2010}. 
It is split-closed, in the $A_\infty$ sense (see \cite[Lemma 5.3]{Seidel2003}).

\begin{defn}
$X$ and $Y$ are said to be \emph{homologically mirror} if there exists an $A_\infty$ quasi-equivalence of $\BbK$-linear, $\Z$-graded, triangulated, split-closed $A_\infty$ categories
\[ \psi \colon \twsplit \EuF(X) \to \dbdg{Y}.\]
\end{defn}

Here, `$\twsplit$' denotes the split-closed triangulated envelope (see \cite[\S 4c]{Seidel2008}).

\subsection{Maximally unipotent monodromy}

We can think of $Y \to \cM$ as a family of $\Bbbk$-schemes parametrized by $\cM$: there is an associated \emph{Kodaira--Spencer map}
\[ \mathsf{KS}_{class} \co T\cM \to H^1(Y, \mathcal{T} Y),\]
where $T\cM := \deriv_\Bbbk \BbK$ is the $\Bbbk$-relative tangent space, and $\mathcal{T}Y$ is the $\BbK$-relative tangent sheaf (see \S \ref{subsec:schdef} for a definition of $\mathsf{KS}_{class}$).

\begin{defn}
\label{defn:maxunip}
We say that $Y \to \cM$ is \emph{maximally unipotent} if 
\[ \mathsf{KS}(\partial_q)^n \neq 0.\]
Here, $n$ is the relative dimension of $Y \to\cM$, and the power is taken with respect to the natural product on the tangential cohomology,
\[ HT^\bullet(Y) := H^\bullet(Y,\wedge\!^\bullet \, \mathcal{T}Y).\]
\end{defn}

\begin{example}
Suppose $\Bbbk = \C$ and $\BbK = \Lambda^\Z = \C\laurent{q}$. 
Let $T$ denote the monodromy of the family $Y$ about $q=0$, acting in the middle algebraic de Rham cohomology $H_{DR}^n(Y/ \cM)$.
By the `monodromy theorem',  one has $T=T_s T_u=T_u T_s$, where $T_s$ has finite order, and  $(T_u-1)^{n+1}=0$. 
Recall that the family is said to have \emph{maximally unipotent monodromy} if $T_s=I$ (so $T=T_u$) and $(T-1)^n\neq 0$ (see, e.g., \cite[\S 5.2]{coxkatz}). 
If the family has maximally unipotent monodromy, then it is maximally unipotent in the sense of Definition \ref{defn:maxunip} (hence the name).
\end{example}

\begin{rmk}
If $X$ and $Y$ are a Calabi-Yau mirror pair (in the sense of Hodge theoretic mirror symmetry \cite{coxkatz}), then $Y$ will always be maximally unipotent. 
Indeed, $T$ is conjugate to $\exp (-2\pi i \Res (\nabla_{d/dq}))$, and so it suffices to show that  $\Res(\nabla_{d/dq})$ is nilpotent of exponent precisely $n$. 
The mirror to this statement is the obvious fact that $[\omega] \in H^\bullet(X;\C)$ is nilpotent of exponent precisely $n$, by nondegeneracy of $\omega$. 
Note that we use the classical (not quantum) product since we are working at the $q=0$ limit (see \cite[\S 8.5.3]{coxkatz}). 
\end{rmk}

\subsection{Main theorem}

\begin{defn}
\label{defn:corehms}
Let $X$ and $Y$ be as in \S \ref{subsec:setup}, and $Y$ be maximally unipotent. 
We say that such $X$ and $Y$ satisfy \emph{core HMS} if there exists a diagram
\[ \xymatrix{ \EuF(X) \ar@{}[d]|-*[@]{\supset} & \dbdg{Y} \ar@{}[d]|-*[@]{\supset} \\
\EuA \ar[r]^{\psi}& \EuB}\]
where
\begin{enumerate}
\item
$\EuB\subset \dbdg{Y}$ is a full subcategory which split-generates;
\item 
$\EuA \subset \EuF(X)$ is a full subcategory; and
\item
$\psi \co \EuA \to \EuB$ is a quasi-equivalence of $A_\infty$ categories.
\end{enumerate}
\end{defn}

For the purposes of the following theorem, we assume that the Fukaya category $\EuF(X)$ has the properties outlined in \S \ref{subsec:fuk}.

\begin{main}\label{gen}
Suppose that $X$ and $Y$ satisfy core HMS. 
Then $\EuA$ split-generates $\EuF(X)$. 
It follows that $X$ and $Y$ are homologically mirror (via a quasi-equivalence extending $\psi$). 
\end{main}

We now record a further result. 
The following definition is from \cite{Ganatra2013}:

\begin{defn}
If the identity $e \in \QH^0(X)$ lies in the image of the open-closed map \eqref{eqn:oc}, then we say that $\EuF(X)$ is \emph{non-degenerate}.
\end{defn}

\begin{main}\label{nondeg}
Let $X$ and $Y$ be as in \S \ref{subsec:setup}, and $Y$ be maximally unipotent. 
If $X$ and $Y$ are homologically mirror, then $\EuF(X)$ is non-degenerate.
\end{main}

The importance of Theorem \ref{nondeg} is that, with some further work, non-degeneracy is sufficient for the closed-open and open-closed maps to be isomorphisms. 
This was proved in \cite{Ganatra2013} in the case of Liouville manifolds, and will be extended in \cite{Ganatra2015} (in preparation) in the case at hand.

Theorems \ref{gen} and \ref{nondeg} will be proved in \S \ref{subsec:pf}. 
The basic idea of the proof of Theorem \ref{gen} is this: check that $\EuA$ satisfies the hypothesis of Theorem \ref{thm:coinj}, by transferring it (via core HMS) to an equivalent hypothesis on $\EuB$, which turns out to be equivalent to maximal unipotence of $Y$. 
The idea for Theorem \ref{nondeg} is similar.

\subsection{The relative and absolute Fukaya categories}

Let us make one remark on potential applications of our results to symplectic topology. 
As we mentioned in \S \ref{subsec:fukint}, when $X$ is equipped with an appropriate divisor $D$ (possibly normal-crossings), one can define the relative Fukaya category $\EuF(X,D)$ \cite{Seidel2002,Sheridan2015,Perutz2015a}: its objects are exact Lagrangian branes in the complement of $D$.
One can also define the absolute Fukaya category $\EuF(X)$ \cite{fooo,Abouzaid2012}. 
Its objects are Lagrangian branes in $X$. 

\begin{conj} 
\label{conj:relabs}
(compare \cite[Assumption 8.1]{Sheridan2015})
There is an embedding of $A_\infty$ categories
\begin{equation}
\label{eqn:relabs}
\EuF(X,D) \subset \EuF(X)
\end{equation}
(possibly after extending the coefficients of $\EuF(X,D)$ to a larger Novikov field). 
The embedding respects open-closed string maps.
\end{conj}

The embedding \eqref{eqn:relabs} appears far from being essentially surjective, as its image consists of Lagrangian branes in $X$ which are exact in the complement of $D$: a very restricted class.

Clearly, the absolute Fukaya category $\EuF(X)$ is more complicated and interesting from the point of view of symplectic topology: it's harder to understand Lagrangians in $X$ than it is to understand exact Lagrangians in $X \setminus D$. 
Nevertheless, observe the following. 
If $\EuA \subset \EuF(X,D)$ satisfies the hypothesis of the split-generation criterion Theorem \ref{thm:abouzaid}, then the image of $\EuA$ under the embedding \eqref{eqn:relabs} will also satisfy the hypothesis of the split-generation criterion (since the embedding respects open-closed maps), and hence split-generate. 

Now suppose that we have established core HMS  for the relative Fukaya category: $\EuA \subset \EuF(X,D)$ is a full $A_\infty$ subcategory which is quasi-equivalent to a split-generating full $\mathsf{dg}$ subcategory $\EuB \subset \dbdg{Y}$.  
Theorem \ref{gen} implies that $\EuA$ split-generates $\EuF(X,D)$: in fact, the proof shows that it satisfies the hypothesis of Theorem \ref{thm:abouzaid}. 
It follows that the image of $\EuA$ under \eqref{eqn:relabs} split-generates $\EuF(X)$, so in fact we have a quasi-equivalence 
\begin{equation}
\label{eqn:abshms} \twsplit\EuF(X) \cong \dbdg{Y}.
\end{equation}
In particular, homological mirror symmetry holds for the absolute Fukaya category $\EuF(X)$, not just for the relative Fukaya category $\EuF(X,D)$. 

Hence, to prove homological mirror symmetry for the absolute Fukaya category \eqref{eqn:abshms}, it suffices to separate the problem into two parts:
\begin{enumerate}
\item \label{it:phms} Prove core HMS for the relative Fukaya category;
\item \label{it:relabs} Prove Conjecture \ref{conj:relabs}.
\end{enumerate}
Part \eqref{it:phms} can be approached by following the blueprint `compute the exact Fukaya category of $X \setminus D$, then solve the deformation problem when one plugs the divisor back in', first outlined in \cite{Seidel2002}. 
This has been carried out for the quartic K3 surface in \cite{Seidel2003} and for higher-dimensional Calabi-Yau hypersurfaces in projective space in \cite{Sheridan2013} (with the caveat that the mirror is a category of \emph{equivariant} coherent sheaves on a scheme with maximally unipotent monodromy in those cases, so some minor alterations to our arguments are necessary). 
In many cases \cite{Sheridan2015,Perutz2015a}, the pseudoholomorphic curve theory involved in part \eqref{it:phms} can be treated using the stabilizing divisor method.

Part \eqref{it:relabs} is a foundational question, about how one sets up one's moduli spaces of pseudoholomorphic curves, and has nothing to do with mirror symmetry. 
If one constructs the relative Fukaya category and the absolute Fukaya category within the same analytic framework, it may be rather trivial (compare the argument in the monotone case \cite{Sheridan2013}). 
However, the relative Fukaya category used in \cite{Sheridan2015,Perutz2015a} is constructed using stabilizing divisors, which while they have the advantage of making the pseudoholomorphic curve theory classical, have the disadvantage of not extending in any straightforward way to give a construction of the absolute Fukaya category (although see \cite{Charest2015}): so to achieve part \eqref{it:relabs}, one would have to relate the stabilizing divisor framework to the Kuranishi space framework of \cite{fooo,Abouzaid2012}, which we have not done.
This is a crucial step if one wants to turn homological mirror symmetry for the relative Fukaya category (such as the result proved in \cite{Sheridan2015}, where part \eqref{it:relabs} was labelled as an assumption with no claim of proof) into homological mirror symmetry for the absolute Fukaya category, and hence say something about the symplectic topology of $X$.

\subsection{Applications}

The case that $X$ is a Calabi-Yau Fermat hypersurface in projective space, and $D$ the intersection of $X$ with the toric boundary of projective space, was considered in \cite{Sheridan2015}. 
Core HMS was proved for a certain full subcategory $\EuA \subset \EuF(X,D)$, consisting of a configuration of Lagrangian spheres in $X \setminus D$. 
Split-generation was proved (based on the assumption, without proof, that the relative and absolute Fukaya categories are related as explained in the previous section), by explicitly computing $\HH^\bullet(\EuA)$ then applying Theorem \ref{thm:coinj}. 
The results in this paper remove the need for this explicit computation of $\HH^\bullet(\EuA)$ (by transferring it to the algebraic geometry side by core HMS, where it is known by the Hochschild-Kostant-Rosenberg isomorphism), and formalize the whole argument.

Core HMS has been proved for the full subcategory $\EuA \subset \EuF(T^*B/T^*B_\Z)$ of Lagrangian sections of a non-singular SYZ torus fibrations with base an integral affine manifold $B$ \cite{Kontsevich2001}. 
The mirror is the dual torus fibration, interpreted as a rigid analytic space rather than as a scheme, and $\EuA$ is mirror to the category of vector bundles (which split-generates the derived category of coherent sheaves, because the mirror space is smooth). 
In the case of abelian varieties, one can prove a similar result, interpreting the mirror instead as an abelian variety \cite{Fukaya2002a}. 
Assuming that the Fukaya category can be shown to satisfy the properties axiomatised in \S \ref{sec:fuk} in this case, our arguments (with appropriate modifications if the mirror is a rigid analytic manifold) should complete these core HMS results to a full proof of homological mirror symmetry (for products of elliptic curves, this was carried out by a different method in \cite{Abouzaid2010d}). 
This allows one to study Lagrangians which are not sections of the torus fibration, in terms of coherent sheaves on the mirror space.

More generally (i.e., allowing for singularities in the Lagrangian torus fibration), a sketch proof of core HMS on the cohomology level is outlined for Gross-Siebert mirror pairs in \cite[Chapter 8]{Aspinwall2009}. 
The subcategory $\EuA$ consists of an infinite family of sections of the Lagrangian torus fibration, which are mirror to the powers $\EuO(r)$ of the ample sheaf $\EuO(1)$ on the mirror variety (which split-generate the derived category of coherent sheaves).  
Assuming the Fukaya category can be shown to satisfy the properties axiomatised in \S \ref{sec:fuk} in this case, and that the sketch proof of core HMS can be turned into an actual proof, our arguments should complete this core HMS result to a full proof of homological mirror symmetry for Calabi-Yau Gross-Siebert pairs.

\subsection{Acknowledgments}

The authors are grateful to David Ben-Zvi, Sheel Ganatra and Nicholas Katz for helpful conversations. 
N.S. is grateful to the IAS and the Instituto Superior T\'{e}cnico for hospitality while working on this project.

\section{The Fukaya category}
\label{sec:fuk}

Let $X$ and $\BbK$ be as in \S \ref{subsec:setup}.
We will give a list of properties that we need the Fukaya category $\EuF(X)$ to have in order for our results to work. 

\subsection{Quantum cohomology}
\label{subsec:qh}

Quantum cohomology $\QH^\bullet(X) := H^\bullet(X;\BbK)$ as a $\BbK$-vector space; the grading is the standard one. 
It is equipped with the graded, $\BbK$-linear quantum cup product $\star$, defined by counting pseudoholomorphic spheres $u\co \mathbb{CP}^1 \to X$, weighted by $q^{\omega(u)} \in \BbK$. 
It is associative and supercommutative, and the identity element $e \in H^\bullet(X;\BbK)$ is also an identity for $\star$. 
It is a Frobenius algebra with respect to the Poincar\'{e} pairing:
\[ \langle \alpha \star \beta,\gamma \rangle = \langle \alpha, \beta \star \gamma \rangle.\]

\begin{rmk}
These properties have been established for the quantum cohomology of a semipositive symplectic manifold in, for example, \cite{mcduffsalamon}.
\end{rmk}

\subsection{Fukaya category}
\label{subsec:fuk}

The Fukaya category $\EuF(X)$ is a $\Z$-graded, $\BbK$-linear, cohomologically unital, proper $A_\infty$ category (in particular, it has no curvature: $\mu^0 = 0$). 
Henceforth in this section we will abbreviate it by $\EuF$.

\begin{rmk}
When $X$ is a Liouville manifold, the Fukaya category was constructed in \cite{Seidel2008}. 
In the completely general case, a construction of the Fukaya category allowing for a single Lagrangian object was given in \cite{fooo}. 
\end{rmk}

\subsection{Closed-open string map}
\label{subsec:co}

There is a graded map of $\BbK$-algebras,
\[ \EuC\EuO \co \QH^\bullet(X) \to \HH^\bullet(\EuF).\]

For any object $L$ of $\EuF$, there is a map of $\BbK$-algebras $\HH^\bullet(\EuF) \to \mathrm{Hom}^\bullet_\EuF(L,L)$; composing $\EuC\EuO$ with this map yields a graded map of $\BbK$-algebras, which we denote by
\[ \EuC\EuO^0 \co \QH^\bullet(X) \to \mathrm{Hom}^\bullet_\EuF(L,L).\]
This map is unital (the map $\EuC\EuO$ ought also to be unital, but we don't need that).

\begin{rmk}
The idea that there should exist an algebra isomorphism between $\QH^\bullet(X)$ and $\HH^\bullet(\EuF)$ goes back to \cite{Kontsevich1994}. 
$\EuC\EuO$ was constructed (under the name $\mathfrak{q}_{1,k}$, and allowing only for a single Lagrangian) in \cite[\S 3.8]{fooo}. 
The conjecture that it ought to be an algebra homomorphism is mentioned in \cite[\S 6]{Fukaya2010d}.
When $X$ is a Liouville manifold, the construction of $\EuC\EuO$ (and the fact that it ought to be a homomorphism of Gerstenhaber algebras) was explained in \cite{Seidel2002}. 
In this context it is a map from symplectic cohomology to Hochschild cohomology of the wrapped Fukaya category (see also \cite{Ganatra2013}). 
\end{rmk}
 
\subsection{Open-closed string map}
\label{subsec:oc}

There is a graded map of $\QH^\bullet(X)$-modules
\[ \EuO\EuC \co \HH_\bullet(\EuF) \to \QH^{\bullet+n}(X),\]
where $\HH_\bullet(\EuF)$ acquires its $\QH^\bullet(X)$-module structure via the map $\EuC\EuO$, and its natural $\HH^\bullet(\EuF)$-module structure.

\begin{rmk}
$\EuO\EuC$ was constructed (under the name $\mathfrak{p}$, and allowing only for a single Lagrangian) in \cite[\S 3.8]{fooo}. 
When $X$ is a Liouville manifold, $\EuO\EuC$ was constructed (from Hochschild homology of the wrapped Fukaya category to symplectic cohomology) in \cite{Abouzaid2010a}. 
It was proved to be a module homomorphism in the same setting in \cite{Ganatra2013} and (in the convex monotone case) in \cite{Ritter2012}. 
\end{rmk}

\subsection{Weak proper Calabi-Yau structure}
\label{subsec:wcyc}

We define $[\phi] \in \HH_n(\EuF)^\vee$ by
\[ [\phi](\alpha) := \langle \EuO\EuC(\alpha), e \rangle.\]
$[\phi]$ is an \emph{$n$-dimensional weak proper Calabi-Yau structure} on $\EuF$: that is, the pairing
\[ \mathrm{Hom}^\bullet_\EuF(K,L) \otimes \mathrm{Hom}^{n-\bullet}_\EuF(L,K) \overset{[\mu^2]}{\to} \mathrm{Hom}^n_\EuF(K,K) \to \HH_n(\EuF) \overset{[\phi]}{\to} \BbK \]
is non-degenerate. 
As a consequence, it induces an isomorphism
\begin{equation}
\label{eqn:hhdual}
\HH^\bullet(\EuF) \to \HH_\bullet(\EuF)^\vee[-n]
\end{equation}
that sends $\alpha \mapsto [\phi] \cap \alpha$ (see, e.g., \cite[Lemma A.2]{Sheridan2013}). 
 
The closed-open and open-closed string maps respect the induced duality map \eqref{eqn:hhdual}, in the sense that the following diagram commutes:
\[ \xymatrixcolsep{5pc}\xymatrix{
	\QH^\bullet(X) \ar[r]^-{\alpha \mapsto \langle \alpha,-\rangle} \ar[d]^{\EuC\EuO} & \QH^\bullet(X)^\vee[-2n] \ar[d]^{\EuO\EuC^\vee} \\
	\HH^\bullet(\EuF) \ar[r]^-{\eqref{eqn:hhdual}} & \HH_\bullet(\EuF)^\vee[-n].} \]

\begin{rmk}
There is a notion of `strict cyclicity' of an $A_\infty$ category, which is strictly stronger than a weak proper Calabi-Yau structure; the Fukaya endomorphism $A_\infty$ algebra of a single Lagrangian was shown to be strictly cyclic in \cite{Fukaya2010}. 
The construction of the weak proper Calabi-Yau structure $[\phi]$ was outlined for the exact Fukaya category in \cite[\S 12j]{Seidel2008}, see also \cite[\S 5]{Seidel2010c}.
\end{rmk}

\subsection{Coproduct}
\label{subsec:cop}

Let $\mathcal{Y}^l_K$ denote the left Yoneda module over $\EuF$ corresponding to an object $K$, let $\mathcal{Y}^r_K$ denote the right Yoneda module over $\EuF$ corresponding to $K$, and let $\EuF_\Delta$ denote the diagonal $(\EuF,\EuF)$ bimodule. 
The coproduct is a morphism of $(\EuF,\EuF)$ bimodules,
\[\Delta \co \EuF_\Delta \to \mathcal{Y}^l_K \otimes_\BbK \mathcal{Y}^r_K.\]

\begin{rmk}
When $X$ is a Liouville manifold, the coproduct was constructed (for the wrapped Fukaya category) in \cite{Abouzaid2010a}. 
\end{rmk}

\subsection{Cardy relation}
\label{subsec:cardy}

The diagram
\[ \xymatrix{
	\HH_\bullet(\EuF)[n] \ar[r]^-{\EuO\EuC} \ar[d]^-{\HH_\bullet(\Delta)} & \QH^\bullet(X) \ar[d]^-{\EuC\EuO^0} \\
	\HH_\bullet(\mathcal{Y}^l_K \otimes_\BbK \mathcal{Y}^r_K) \ar[r]^-{H^*(\mu)} & \mathrm{Hom}^\bullet_\EuF(K,K) } \]
commutes up to a sign $(-1)^{n(n+1)/2}$.

\begin{rmk}
When $X$ is a Liouville manifold, this version of the Cardy relation was proved (for the wrapped Fukaya category) in \cite{Abouzaid2010a}.
\end{rmk}

\subsection{Kodaira--Spencer maps}
\label{subsec:defmap}

We recall the definition of the \emph{categorical Kodaira--Spencer map} for a $\BbK$-linear $A_\infty$ category $\EuC$: 
\begin{eqnarray*}
\mathsf{KS}_{cat} \co \deriv_\Bbbk \BbK & \to & \HH^2(\EuC)\\
\mathsf{KS}_{cat}(\xi) & := & \xi(\mu^*),
\end{eqnarray*}
where $\mu^*$ denotes the $A_\infty$ structure maps, written with respect to a choice of $\BbK$-basis for each morphism space in $\EuC$ (see \cite[\S 3.5]{Sheridan2015a}; this class is closely related to the \emph{Kaledin class} \cite{Kaledin2007,Lunts2010b}).

We have
\begin{equation} 
\label{eqn:COKS} 
\EuC\EuO ([\omega])  = \mathsf{KS}_{cat}(q \partial_q) \in \HH^2(\EuF),
\end{equation}
where $[\omega] \in \QH^2(X)$ is the class of the symplectic form.

\section{Split-generation}

\subsection{Abouzaid's argument}
\label{subsec:gen}

Assume that the results of Sections \ref{subsec:qh}, \ref{subsec:fuk}, \ref{subsec:co}, \ref{subsec:oc}, \ref{subsec:cop} and \ref{subsec:cardy} hold. 

\begin{thm}[Theorem \ref{thm:abouzaid}] 
Let $\EuA$ be a full subcategory of $\EuF(X)$. 
If the identity $e \in \QH^0(X)$ lies in the image of the map
\begin{equation}
\label{eqn:oc02}
\EuO\EuC|_\EuA \co \HH_{-n}(\EuA) \to \QH^0(X),
\end{equation}
then $\EuA$ split-generates $\EuF(X)$.
\end{thm}
\begin{proof}
The proof is identical to that of \cite[Theorem 1.1]{Abouzaid2010a}.
\end{proof}

Now assume that the results of \S \ref{subsec:wcyc} hold. 

\begin{thm}[Theorem \ref{thm:coinj}] 
Let $\EuA$ be a full subcategory of $\EuF(X)$. 
If the map
\begin{equation}
\label{eqn:co2n}
 \EuC\EuO|_\EuA \co \QH^{2n}(X) \to \HH^{2n}(\EuA)
\end{equation}
is injective, then $\EuA$ split-generates $\EuF(X)$.
\end{thm}
\begin{proof}
By \S \ref{subsec:wcyc}, \eqref{eqn:co2n} is dual to \eqref{eqn:oc02}. 
In particular, if the former is injective, the latter is surjective, hence contains the identity $e$ in its image. 
The result follows by Theorem \ref{thm:abouzaid}.
\end{proof}

\subsection{The Kodaira--Spencer map and split-generation}

Now assume that the results of \S \ref{subsec:defmap} hold. 

\begin{defn}
\label{defn:npot}
We say that a $\BbK$-linear $A_\infty$ category $\EuA$ is \emph{$n$-potent} if 
\[ \mathsf{KS}_{cat}(\partial_q)^{\cup n} \neq 0,\]
where $\cup$ denotes the Yoneda product on $\HH^\bullet(\EuA)$.
\end{defn}

\begin{thm}
\label{thm:npotgen}
Let $X$ be as in \S \ref{subsec:setup}: connected, Calabi-Yau, and $2n$-dimensional; and let $\EuA \subset \EuF(X)$ be a full subcategory. 
If $\EuA$ is $n$-potent, then $\EuA$ split-generates $\EuF(X)$.
\end{thm}
\begin{proof}
By the results of \S \ref{subsec:defmap},
\[ \EuC\EuO([\omega]) =  \mathsf{KS}_{cat}(q\partial_q) \in \HH^2(\EuA) .\]
Because $\EuC\EuO$ is an algebra homomorphism by \S \ref{subsec:co}, 
\begin{equation}
\label{eqn:coomegan}
 \EuC\EuO\left([\omega]^{\star n} \right) = \mathsf{KS}_{cat}(q\partial_q)^{\cup n} \in \HH^{2n}(\EuA).
\end{equation}
Because $\EuA$ is $n$-potent, this class is non-zero. 
Because $X$ is connected and $2n$-dimensional, $\QH^{2n}(X)$ has rank $1$, so the fact that \eqref{eqn:coomegan} is non-zero implies that \eqref{eqn:co2n} is injective. 
The result follows by Theorem \ref{thm:coinj}.
\end{proof}

\begin{rmk}
Theorem \ref{thm:npotgen} does not hold if we violate our standing assumption that $\EuF(X)$ is $\Z$-graded, which can only be expected when $X$ is Calabi-Yau (for example, when $X$ is monotone, the best one can hope for is that the grading group is $\Z/2N$). 
That is because the proof crucially uses the fact that $\QH^{2n}$ is $1$-dimensional, which need not hold for other grading groups. 
However, Theorem \ref{thm:coinj} may still be applied (compare \cite{Sheridan2013}).
\end{rmk}

\section{Hochschild--Kostant--Rosenberg}

\subsection{The HKR isomorphism}

Let $Y \to \cM$ be as in \S \ref{subsec:setup}. 
The \emph{tangential cohomology} of $Y$ is defined to be the cohomology of the sheaf of polyvector fields:
\[ HT^\bullet(Y) := \bigoplus_{p+q = \bullet}H^p(Y,\wedge^q \, \mathcal{T}Y);\]
it is a graded $\BbK$-algebra, via wedge product of polyvector fields. 
Swan \cite{Swan1996} defines the Hochschild cohomology of $Y$ to be
\[ \HH^\bullet(Y) := \ext^\bullet_{Y \times Y}(\Delta_* \mathcal{O}_Y,\Delta_* \mathcal{O}_Y),\]
where $\Delta \co Y \to Y \times Y$ is the diagonal embedding. 
It is a graded $\BbK$-algebra, via the Yoneda product. 

The Hochschild--Kostant--Rosenberg isomorphism \cite{Hochschild1962,Gerstenhaber1988b,Swan1996,Yekutieli2002} is an
explicit quasi-isomorphism
\[ \Delta\!^* \Delta_* \mathcal{O}_Y \to \bigoplus_q \Omega^q_Y[q],\]
which induces an isomorphism
\begin{equation}
\label{eqn:hkr}
 \HKR \co HT^\bullet(Y) \to \HH^\bullet(Y)
\end{equation}
(see \cite[Corollary 4.2]{Caldararu2005}).

There is also an isomorphism \cite{Lowen2005,Toen2006a}
\begin{equation}
\label{eqn:hhcats}
 \HH^\bullet(Y) \cong \HH^\bullet(\dbdg{Y}).
\end{equation}
Composing the two yields an isomorphism
\[ \HKR_{cat}:HT^\bullet(Y) \to \HH^\bullet(\dbdg{Y}),\]
which we call the \emph{categorical HKR isomorphism}.

\subsection{The categorical HKR map and deformation theory}

We have the classical Kodaira--Spencer map (we recall the definition in Section \ref{subsec:schdef}) 
\[ \mathsf{KS}_{class} \co \deriv_\Bbbk \BbK \to H^1(Y, \mathcal{T} Y) \subset HT^2(Y),\]
and the categorical Kodaira--Spencer map \cite[\S 3.5]{Sheridan2015a}
\[ \mathsf{KS}_{cat} \co \deriv_\Bbbk \BbK \to \HH^2(\dbdg{Y}).\]

These classes are related in the expected way:

\begin{prop}
\label{prop:KSHKR}
We have 
\[ \HKR_{cat} \circ \mathsf{KS}_{class} = \mathsf{KS}_{cat}.\]
\end{prop}

We were not able to locate a proof of this result in the literature, although the statement will surprise no-one: we present a proof in Appendix \ref{app:kshkr}.

\subsection{The twisted HKR map}

This isomorphism $\HKR$ does not respect the algebra structures: this can be `corrected' by twisting by the square root of the Todd class. 
Thus, one defines
\begin{eqnarray*}
I^* \co HT^\bullet(Y) & \to & \HH^\bullet(Y) \\
I^*(\alpha) &:= & \mathsf{HKR}\left(\mathrm{td}_Y^{1/2} \wedge \alpha \right).
\end{eqnarray*}
The map $I^*$ respects the algebra structure (see \cite[Corollary 1.5]{Calaque2010}; this result was first claimed in \cite[Section 8.4]{Kontsevich2003}, see also \cite[Claim 5.1]{Caldararu2005}). 

The isomorphism \eqref{eqn:hhcats} respects the algebra structure: so composing it with $I^*$ yields an algebra isomorphism
\[ I^*_{cat} \co HT^\bullet(Y) \to \HH^\bullet(\dbdg{Y});\]
this should be regarded as the mirror to the closed--open map $\EuC\EuO$.

\begin{cor}
\label{cor:ytriv}
When $Y$ has trivial canonical sheaf, 
\begin{equation}
\label{eqn:KSIcat}
I^*_{cat} \circ \KS_{class} = \KS_{cat}.
\end{equation}
\end{cor}
\begin{proof}
When $Y$ has trivial canonical sheaf, the degree-$2$ component of $\mathrm{td}_Y^{1/2}$ vanishes; so the maps
\[ \mathsf{HKR},I^* \co H^1(\mathcal{T}Y) \to \HH^2(Y)\]
coincide. 
The result then follows immediately from Proposition \ref{prop:KSHKR}.
\end{proof}

\subsection{Proofs of Theorems \ref{gen} and \ref{nondeg}}
\label{subsec:pf}

Let $X$ and $Y$ be as in \S \ref{subsec:setup}, and $Y$ be maximally unipotent.

\begin{thm}[Theorem \ref{nondeg}]
If such $X$ and $Y$ are homologically mirror, then $\EuF(X)$ is non-degenerate.
\end{thm}
\begin{pf}
Because $Y$ is maximally unipotent, $\KS_{class}(\partial_q)^n \neq 0$. 
Because $Y$ has trivial canonical sheaf by hypothesis, we have
\[ I^*_{cat}(\KS_{class}(\partial_q)) = \KS_{cat}(\partial_q)\]
by Corollary \ref{cor:ytriv}. 
Because $I^*_{cat}$ is an algebra isomorphism, it follows that
\[ 0 \neq I^*_{cat}\left(\KS_{class}(\partial_q)^{n}\right) = \KS_{cat}(\partial_q)^{\cup n}.\]
In particular, $\dbdg{Y}$ is $n$-potent (Definition \ref{defn:npot}). 

Because Hochschild cohomology and the categorical Kodaira--Spencer map are Morita invariant (see, e.g., \cite[\S 4.4]{Sheridan2015a} for Morita invariance of $\mathsf{KS}_{cat}$), it follows from homological mirror symmetry that $\EuF(X)$ is $n$-potent. 
As in the proof of Theorem \ref{thm:npotgen}, this implies that $\EuC\EuO$ is non-zero in degree $2n$, and hence (dually) that $\EuO\EuC$ is non-zero in degree $0$, and hence that the identity $e \in \QH^0(X)$ is in the image: so $\EuF(X)$ is non-degenerate.
\end{pf}

\begin{thm}[Theorem \ref{gen}]
Suppose that core HMS (Definition \ref{defn:corehms}) holds.  Then $\EuA$ split-generates $\EuF(X)$.
\end{thm}
\begin{proof}
As in the proof of Theorem \ref{nondeg}, one proves that $\dbdg{Y}$ is $n$-potent.
Core HMS requires a split-generating subcategory $\EuB \subset \dbdg{Y}$. 
The restriction map $\HH^\bullet(\dbdg{Y}) \to \HH^\bullet(\EuB)$ is an isomorphism, by Morita invariance of Hochschild cohomology; so $\EuB$ is $n$-potent. 

Core HMS also requires a subcategory $\EuA \subset \EuF(X)$ that is quasi-equivalent to $\EuB$; by Morita invariance of the categorical Kodaira--Spencer map, it follows that $\EuA$ is $n$-potent. 
Hence, by Theorem \ref{thm:npotgen}, $\EuA$ split-generates $\EuF(X)$. 
\end{proof}

\appendix

\section{The Hochschild--Kostant--Rosenberg isomorphism and the Kodaira--Spencer map}
\label{app:kshkr}

The aim of this appendix is to prove the following:

\begin{prop}[Proposition \ref{prop:KSHKR}]
\label{prop:hkronks}
For any $\xi \in \deriv_\Bbbk \BbK$, the isomorphism 
\[ \HKR_{cat} \co HT^\bullet(Y) \to \HH^\bullet(\dbdg{Y}) \]
takes $\mathsf{KS}_{class}(\xi)$ to $\mathsf{KS}_{cat}(\xi)$.
\end{prop}

Let us give a preview of the proof. 
We recall the construction of the isomorphism $\HKR_{cat}$. 
First, we have the HKR isomorphism \cite{Hochschild1962,Swan1996,Yekutieli2002}
\[ \HKR \co HT^\bullet(Y) \to \ext^\bullet_{Y \times Y}(\Delta_* \mathcal{O}_Y,\Delta_* \mathcal{O}_Y).\]  

Next, we have the isomorphism 
\[ \ext^\bullet_{Y \times Y} (\Delta_* \mathcal{O}_Y,\Delta_* \mathcal{O}_Y) \cong  \ext^\bullet_{\algd\bimod}(\algd,\algd)\]
(see \cite[Theorem 3.1]{Swan1996}). 
Here, we fix an open affine cover of $Y$, and $\algd$ is the corresponding diagram of $\BbK$-algebras of functions, in the sense of Gerstenhaber and Schack (see \cite[\S 28]{Gerstenhaber1988b}). 

Next, we have the isomorphism
\[ \ext^\bullet_{\algd\bimod}(\algd,\algd) \cong \ext^\bullet_{\algd!\bimod}(\algd!,\algd!),\]
where $\algd!$ is the `diagram algebra': this is the `special cohomology comparison theorem' of Gerstenhaber and Schack \cite{Gerstenhaber1983,Gerstenhaber1988}. 

Next, we have the isomorphism
\[ \ext^\bullet_{\algd!\bimod}(\algd!,\algd!) \cong \pmb{\mathsf{H}}^\bullet_{\mathrm{ab}}(\algd\modd),\]
where the latter is Lowen and Van den Bergh's Hochschild cohomology of the abelian category $\algd\modd$: see \cite[Theorem 7.2.2]{Lowen2005}.

Finally, we have the isomorphisms
\[ \pmb{\mathsf{H}}^\bullet_{\mathrm{ab}}(\algd\modd) \cong \pmb{\mathsf{H}}^\bullet_{\mathrm{ab}}(Coh(Y)) \cong  \HH^\bullet(\dbdg{Y}),\]
proved in \cite[Corollary 7.7.3 and Theorem 6.1]{Lowen2005}.

For each of these Hochschild cohomology-type algebras, we define a `deformation class' associated to $\xi$ (equal to $\KS_{class}(\xi)$ in $HT^2(Y)$, and to $\KS_{cat}(\xi)$ in  $\HH^2(\dbdg{Y})$), and prove that each isomorphism in the chain respects deformation classes. 
In fact, up until the categories start appearing (with $\pmb{\mathsf{H}}^\bullet_{\mathrm{ab}}(\algd\modd)$), we associate a deformation class to an arbitrary first-order deformation of the scheme $Y$, of which the deformations associated to a derivation of the base are a special case.

\subsection{Deformations of algebras}\label{subsection:deform algebras}

We begin the proof of Proposition \ref{prop:hkronks} with local considerations, based on an account by Bezrukavnikov--Ginzburg \cite{Bezrukavnikov2007}.

Let $\BbK$ be a field, and $\alg$ an associative $\BbK$-algebra. 
The multiplication map
 \[  m_\alg \colon \alg\otimes_\BbK \alg\to \alg   \]
is a surjective map of $(\alg,\alg)$-bimodules. 
Define $I_{\alg} :=\ker  m_\alg$ as an $(\alg,\alg)$-bimodule; it is isomorphic to the space of ($\BbK$-relative) noncommutative 1-forms $\Omega^\nc_\alg$ \cite{Cuntz1995}. 
The map
\[  \mathsf{d}\colon \alg\to \Omega^\nc_\alg,\quad  \mathsf{d} x =   x\otimes 1 - 1\otimes x \]
is the universal noncommutative derivation of $\alg$ (i.e., the universal map $\mathsf{d}' \co  \alg \to B$, where $B$ is an $(\alg,\alg)$ bimodule, satisfying $\mathsf{d}'(x\cdot y) = \mathsf{d}'x \cdot y + x \cdot \mathsf{d}'y$).

If $\alg$ is commutative, then $\Omega_\alg \cong I_\alg/I_\alg^2$ is the space of commutative 1-forms: the induced map $\mathsf{d} \co \alg \to \Omega_\alg$ is the universal commutative derivation of $\alg$  (i.e., the universal map $\mathsf{d}' \co \alg \to B$, where $B$ is an $\alg$-module, satisfying $\mathsf{d}'(x \cdot y) = x \cdot \mathsf{d}'y + y \cdot \mathsf{d}'x$).

\paragraph{The Atiyah class.} The short exact sequence of bimodules
\begin{equation}
\label{eqn:atnc}
 0\to \Omega^\nc_{\alg} \to \alg\otimes \alg\to \alg \to 0
\end{equation}
gives rise to a morphism $  \alg \to \Omega^\nc_\alg[1] $
in the derived category of $(\alg,\alg)$-bimodules: this morphism is called the \emph{noncommutative Atiyah class}, and denoted
\[ \At^\nc_\alg \in \ext^1_{\alg\bimod}(\alg,\Omega^\nc_\alg).\]
If $\alg$ is commutative, we consider the short exact sequence of bimodules
\begin{equation}
\label{eqn:atcomm} 
0 \to \Omega_\alg \to \alg \otimes \alg/I_\alg^2 \to \alg \to 0,
\end{equation}
which gives rise to the \emph{commutative Atiyah class}
\[ \At_\alg \in \ext^1_{\alg \bimod}(\alg,\Omega_\alg).\]

\paragraph{The Kodaira--Spencer class.} 
Let $V$ be a finite-dimensional $\BbK$-vector space, and 
\[\BbK_\varepsilon := \BbK[V]/V^2 \cong \BbK \oplus V^*.\]
Let $\Alg$ be a \emph{$\BbK_\varepsilon$-deformation of $\alg$}, i.e., a free $\BbK_\varepsilon$-algebra equipped with a $\BbK$-algebra isomorphism $\Alg \otimes_{\BbK_\varepsilon} \BbK \cong \alg$. 
Thus one has a surjective homomorphism $f\colon\Alg \to \alg$ with kernel isomorphic to $ V^*\otimes \alg$, and squaring to $0$.
This gives rise to a short exact sequence of $(\alg,\alg)$ bimodules \cite[Corollary 2.11]{Cuntz1995}
\begin{equation}\label{differentials}
 0\to V^* \otimes \alg \xrightarrow{ 1\otimes \mathsf{d} \otimes 1}    \alg \otimes_\Alg \Omega^\nc_\Alg \otimes_\Alg \alg \xrightarrow{f \otimes f} \Omega^\nc_\alg \to 0.  
 \end{equation}
This yields a map $\Omega^\nc_\alg \to V^*\otimes \alg [1]$ in the derived category of $(\alg,\alg)$-bimodules: this morphism is called the \emph{noncommutative Kodaira--Spencer class}, and denoted
\[ \theta^\nc_\Alg \in \ext^1_{\alg\bimod}(\Omega^\nc_\alg,V^*\otimes \alg). \]

Now we consider the case that $\alg$ and $\Alg$ are commutative: the analogue of \eqref{differentials} is the conormal short exact sequence of $\alg$-modules
\begin{equation}
\label{eqn:commKS}
 0 \to V^* \otimes \alg \to \alg \otimes_\Alg \Omega_\Alg \to \Omega_\alg \to 0,
\end{equation}
which yields the \emph{commutative Kodaira--Spencer class}
\[ \theta_\Alg \in \ext^1_{\alg\modd}(\Omega_\alg,V^* \otimes \alg).\]

\paragraph{The deformation class.}
Bezrukavnikov and Ginzburg define the \emph{deformation class} of $\Alg$ to be the composition of the Atiyah and Kodaira--Spencer classes in the derived category of $(\alg,\alg)$-bimodules:
\[ \deform_\alg(\Alg) := \theta_\Alg \circ \At_\alg \in \ext^2_{\alg\bimod}(\alg,V^* \otimes \alg).\]
It is represented by the 2-extension that is the splicing of the extensions defining $\At_\alg$ and $\theta_\Alg$:
\begin{equation}
\label{eqn:2ext}
 0 \to V^* \otimes \alg \to \alg \otimes_\Alg \Omega^\nc_\Alg \otimes_\Alg \alg \to \alg \otimes \alg \to \alg \to 0.
\end{equation}

When $\alg$ and $\Alg$ are commutative, we have an alternative description of the deformation class, in terms of the commutative Atiyah and Kodaira--Spencer classes. 
Namely, we apply the obvious exact functor $\alg\modd \to \alg\bimod$ to the short exact sequence \eqref{eqn:commKS}, and splice it with the short exact sequence \eqref{eqn:atcomm} to obtain the 2-extension
\begin{equation}
\label{eqn:2extcomm}
0 \to V^* \otimes \alg \to \alg \otimes_\Alg \Omega_\Alg \to \alg \otimes \alg/I_\alg^2 \to \alg \to 0.
\end{equation}
There is an obvious homomorphism of exact sequences of $(\alg,\alg)$ bimodules from  \eqref{eqn:2ext} to \eqref{eqn:2extcomm}, equal to the identity on both ends: so  these 2-extensions give rise to the same class in $ \ext_{\alg\bimod}^2(\alg,\alg)$. 
In particular, the 2-extension \eqref{eqn:2extcomm} is an alternative description of the deformation class, valid in the commutative case.

\paragraph{Cocycle representing the deformation class.}

Classically, the deformation class was defined by giving an explicit Hochschild cocycle associated to the deformation. 
To obtain a cochain complex computing $\ext^2_{\alg\bimod}(\alg,V^* \otimes \alg)$, we replace the diagonal bimodule by its bar resolution $\EuB_\bullet(\alg)$, where $\EuB_q(\alg) = \alg^{\otimes q+2}$. 
This gives the Hochschild cochain complex 
\begin{eqnarray*} 
CC^\bullet(\alg,V^*\otimes\alg) & := & \Hom_{\alg\bimod}(\EuB_\bullet(\alg),V^* \otimes \alg) \\
&\cong & \prod_{q \ge 0} \Hom_\BbK \left(\alg^{\otimes q},V^* \otimes \alg\right)
\end{eqnarray*}
equipped with the Hochschild differential $\delta$.

To obtain a cocycle in this complex that represents $\deform_\alg(\Alg)$, one chooses a $\BbK$-vector space splitting
\[ \xymatrix{ \Alg \ar@<0.5ex> @{->>}[r] & \alg \ar@<0.5ex> @{-->}[l]^s},\]
which induces an isomorphism of $\BbK$-vector spaces
\[ \Alg \cong \alg \oplus V^* \otimes \alg.\]

The multiplication then takes the form
\begin{equation}\label{eqn:deform mult}
 m_\Alg(a \oplus v \cdot b,c \oplus w \cdot d)  = m_\alg(a,c) \oplus \beta_\Alg(a,c) + v \cdot m_\alg(b,c) + w \cdot m_\alg(a,d). 
\end{equation}
The map $\beta_\Alg\colon \alg\otimes \alg \to V^*\otimes \alg$ is a Hochschild 2-cocycle (because $m_\Alg$ is associative), hence defines a class
\[  [\beta_\Alg] \in \ext^2_{\alg\bimod}(\alg,V^* \otimes \alg). \]

Bezrukavnikov and Ginzburg state that this class coincides with their definition of the deformation class $\deform_\alg(\Alg)$: we write down a proof, since we will want to extend the proof to a more general setting in the next section.

\begin{lem}
\label{lem:betadef}
We have $[\beta_\Alg] = \deform_\alg(\Alg)$.
\end{lem}
\begin{proof}
The 2-extension \eqref{eqn:2ext} defining $\deform_\alg(\Alg)$ gives rise to a morphism $\alg \to V^* \otimes  \alg[2]$ in the derived category of $(\alg,\alg)$ bimodules via the following diagram:
\[ \xymatrix{ && 0 \ar[r] & \alg \ar[r] & 0\\
 0 \ar[r] & V^* \otimes \alg \ar[r]  \ar[d]^{\id} & \alg \otimes_\Alg \Omega^\nc_\Alg \otimes_\Alg \alg \ar[r] & \alg \otimes \alg \ar[r] \ar[u]^{m_\alg} & 0 \\
0 \ar[r] & V^* \otimes \alg \ar[r] & 0. &&}\]
Namely, the map $m_\alg$ (viewed as a map of complexes as indicated by the diagram) is a quasi-isomorphism by exactness of \eqref{eqn:2ext}, hence can be inverted in the derived category, so we obtain the morphism $\alg \to V^* \otimes \alg[2]$. 

On the other hand, we can replace $\alg$ with $\EuB_\bullet(\alg)$, and write down an explicit chain map inverting the quasi-isomorphism $m_\alg$:
\[  \xymatrix{ \ldots \ar[r] & \alg \otimes \alg \otimes \alg \otimes \alg \ar[r] \ar[d]^{\beta_\Alg}& \alg \otimes \alg \otimes \alg \ar[r] \ar[d]^{ds}& \alg \otimes \alg \ar[r] \ar[d]^{\id} & 0 \\
 0 \ar[r] & V^* \otimes \alg \ar[r]  & \alg \otimes_\Alg \Omega^\nc_\Alg \otimes_\Alg \alg \ar[r] & \alg \otimes \alg \ar[r] & 0.}\]
Here, 
\[ ds \co a \otimes b \otimes c \mapsto a \otimes ds(b) \otimes c,\]
and
\[ \beta_\Alg \co a \otimes b \otimes c \otimes d \mapsto a \cdot \beta_\Alg(b,c) \cdot d.\] 
One easily verifies that this is a chain map (for this, it is helpful to observe that
\[\beta_\Alg(a,b) = s(a) \cdot s(b) - s(a\cdot b),\]
as follows immediately from \eqref{eqn:deform mult}). 
It also inverts $m_\alg$: this follows from the fact that the augmentation of the bar resolution is also given by $m_\alg$. 
Hence, composing this quasi-inverse with the obvious map to $V^* \otimes \alg$, we find the corresponding Hochschild cochain to be $\beta_\Alg$, as required.
\end{proof}

\subsection{Deformations of diagrams}

We recall the notion of a \emph{diagram of algebras}, following Gerstenhaber and Schack \cite[\S 17]{Gerstenhaber1988b} (whose notation and conventions we adopt). 
Let $\cB$ be a poset: we regard it as a category, with a unique map $i \to j$ if $i \le j$, and no other maps. 
If $v$ is a map in $\cB$, we denote the domain by $dv$ and the codomain by $cv$: so $v \in \cB(cv,dv)$. 

A \emph{diagram} over $\cB$ is a contravariant functor $\algd \co \cB^{op} \to \kalg$.
We write $\algd^i$ for $\algd(i)$, and $\varphi^v$ for $\algd(v)$: so $\varphi^v \co  \algd^{cv} \to \algd^{dv}$ is a $\BbK$-algebra homomorphism. 

\paragraph{Atiyah class.}
Gerstenhaber and Schack define an abelian category $\algd\bimod$ of $(\algd,\algd)$-bimodules, which has enough projectives and injectives. 
Applying the construction from the previous section locally, we have a short exact sequence
\begin{equation}
\label{eqn:atd}
 0 \to \Omega^\nc_\algd \to \algd \otimes \algd \to \algd \to 0,
\end{equation}
which defines a map $\algd \to \Omega^\nc_\algd[1]$ in the derived category of $(\algd,\algd)$-bimodules; the Atiyah class is the corresponding class
\[ \At_\algd \in \ext^1_{\algd\bimod}(\algd,\Omega^\nc_\algd).\]

\paragraph{Kodaira--Spencer class.} Now, let $\BbK_\varepsilon = \BbK[V]/V^2$ be as in the previous section. 
A \emph{deformation} of $\algd$ over $\BbK_\varepsilon$ is a diagram $\Algd$ of $\BbK_\varepsilon$-modules over $\cB$, equipped with an isomorphism $\Algd \otimes_{\BbK_{\varepsilon}} \BbK \cong \algd$: so each $\Algd^i$ is a deformation of $\algd^i$ over $\BbK_\varepsilon$.
Again, we can apply the construction of the previous section locally to obtain a short exact sequence
\begin{equation}
\label{eqn:ksd}
0\to V^* \otimes \algd \to    \algd \otimes_\Algd \Omega^\nc_\Algd \otimes_\Algd \algd \to \Omega^\nc_\algd \to 0,
\end{equation}
which defines a map $\Omega^\nc_\algd \to V^* \otimes \algd[1]$ in the derived category of $(\algd,\algd)$-bimodules; the Kodaira--Spencer class is the corresponding class
\[ \theta_\Algd \in \ext^1_{\algd\bimod}(\Omega^\nc_\algd,V^* \otimes \algd).\]

\paragraph{Deformation class.}
We define the deformation class $\deform_\algd(\Algd):= \theta_\Algd \circ \At_\algd$ as before: it corresponds to the 2-extension 
\begin{equation}
\label{eqn:defd}
 0 \to V^* \otimes \algd \to \algd \otimes_\Algd \Omega^\nc_\Algd \otimes_\Algd \algd \to \algd \otimes \algd \to \algd \to 0.
\end{equation}
We regard it as an element of $\Hom(V,\ext^2_{\algd\bimod}(\algd,\algd))$. 
If $\algd$ and $\Algd$ are commutative, then it also corresponds to the 2-extension
\begin{equation}
\label{eqn:defdcomm}
 0 \to V^* \otimes \algd \to \algd \otimes_\Algd \Omega_\Algd \to \algd \otimes \algd/I_\algd^2 \to \algd \to 0.
\end{equation}

\paragraph{Cocycle representing the deformation class.}
In the case of a single algebra, the Hochschild cochain complex $CC^\bullet(\alg,V^* \otimes \alg)$ has cohomology $\ext^\bullet_{\alg\bimod}(\alg,V^* \otimes \alg)$, because the bar resolution $\EuB_\bullet(\alg)$ is a projective resolution of the diagonal $(\alg,\alg)$-bimodule. 
For a diagram of algebras, the bar resolution $\EuB_\bullet(\algd^i)$ is \emph{locally projective} (i.e., projective as an $(\algd^i,\algd^i)$-bimodule for all $i$), but not \emph{projective} as an $(\algd,\algd)$-bimodule. 
To obtain a projective resolution of $\algd$, one must use Gerstenhaber and Schack's \emph{generalized simplicial bar resolution}, denoted $\EuS_\bullet\EuB_\bullet(\algd)$ \cite[\S 20]{Gerstenhaber1988b}. 

The resulting cochain complex is (in the notation of \cite[\S 21]{Gerstenhaber1988b})
\begin{eqnarray*}
CC^r(\algd,V^* \otimes\algd) &:=& \bigoplus_{p+q=r}\Hom_{\algd\bimod}(\EuS_p\EuB_q(\algd),V^* \otimes \algd) \\
&\cong& \bigoplus_{p+q=r}\prod_{\dim \sigma = p} CC^q(\algd^{c\sigma},|V^* \otimes \algd^{d\sigma}|_{|\sigma|}),
\end{eqnarray*}
where the product is over all $p$-dimensional simplices $\sigma$ in the poset $\cB$.

Given a deformation of diagrams $\Algd$, and a splitting 
\[ \xymatrix{ \Algd^i \ar@<0.5ex> @{->>}[r] & \algd^i \ar@<0.5ex> @{-->}[l]^{s^i}}\]
for all $i$ (the $s^i$ can be chosen independently), Gerstenhaber and Schack define a cochain $\beta_\Algd \in CC^2(\algd,V^*\otimes\algd)$. 
Namely, $\beta_\Algd$ has a component for each $0$-simplex
\begin{eqnarray*}
\beta_\Algd^{(i)} & \in &CC^2(\algd^i,V^* \otimes \algd^i), \\
\beta_\Algd^{(i)}(a,b) &:=& s^i(a) \cdot s^i(b) - s^i(a \cdot b),
\end{eqnarray*}
and a component for each $1$-simplex, i.e., for each morphism $v \in \cB(cv,dv)$: 
\begin{eqnarray*}
\beta_\Algd^{(v)} & \in & CC^1(\algd^{cv},V^*\otimes \algd^{dv}), \\
\beta_\Algd^{(v)}(a) &=& s^{dv} \circ \algd(v) - \Algd(v) \circ s^{cv}.
\end{eqnarray*}

\begin{lem}
\label{lem:betadefd}
We have $[\beta_\Algd] = \deform_\algd(\Algd)$.
\end{lem}
\begin{proof}
We follow the proof of Lemma \ref{lem:betadef}. 
We replace $\algd$ by its projective resolution $\EuS_\bullet\EuB_\bullet(\algd)$, and construct a morphism 
\begin{equation}
\label{eqn:morphSBa}
\EuS_\bullet\EuB_\bullet(\algd) \to \{0 \to V^* \otimes \algd \to \algd \otimes \Omega^\nc_{\Algd} \otimes \algd \to \algd \otimes \algd \to 0\}
\end{equation}
whose composition with the augmentation from the right-hand side to $\algd$ is the augmentation of the generalized simplicial bar resolution. 

Giving a morphism \eqref{eqn:morphSBa} is equivalent to giving a cocycle
\[ \beta  \in  CC^0(\algd,0 \to V^*\otimes \algd \to \algd\otimes\Omega^\nc_\Algd\otimes\algd \to \algd \otimes \algd \to 0).\]
As proven in \cite[\S 21]{Gerstenhaber1988b}, this cochain complex is isomorphic to
\begin{eqnarray} 
\label{eqn:ccdim0}&&\prod_{i} CC^2(\algd^i,V^* \otimes \algd^i) \oplus CC^1(\algd^i,\algd^i \otimes\Omega^\nc_{\Algd^i} \otimes \algd^i) \oplus CC^0(\algd^i,\algd^i \otimes \algd^i) \\
\label{eqn:ccdim1}& \oplus& \prod_{i \overset{v}{\to} j} CC^1(\algd^j,V^* \otimes \algd^i) \oplus CC^0(\algd^j,\algd^i \otimes\Omega^\nc_{\Algd^i} \otimes \algd^i) \\
\label{eqn:ccdim2}& \oplus & \prod_{i \overset{v}{\to} j \overset{w}{\to} k} CC^0(\algd^k,V^* \otimes \algd^i).
\end{eqnarray}
The differential on this complex has three components:
\[ \delta = \delta' + \delta'' + \delta''',\]
where $\delta'$ is the differential in the simplicial direction (it increases $p$), $\delta''$ is the Hochschild differential (it increases $q$), and $\delta'''$ is composition with the differential in the complex $0 \to V^* \otimes \algd \to \algd \otimes \Omega^\nc_\Algd \otimes \algd \to \algd \otimes \algd \to 0$.

The morphism $\beta$ we construct has a component $\beta^{0,2} \oplus \beta^{0,1} \oplus \beta^{0,0}$ in \eqref{eqn:ccdim0}: it coincides with the construction in the proof of Lemma \ref{lem:betadef}, applied to the individual deformations $\Algd^i$ with splittings $s^i$ (in particular, $\beta^{0,2}$ is the product of the cochains $\beta^{(i)}_\Algd$ defined above).
It also has a component $\beta^{1,1} \oplus 0$ in \eqref{eqn:ccdim1}: this is the product of the cochains $\beta^{(v)}_\Algd$ defined above. 
The component $\beta^{2,0}$ in \eqref{eqn:ccdim2} vanishes.

To prove that this $\beta$ is a cocycle, we must show that
\[ (\delta'+\delta''+\delta''')(\beta^{0,2} + \beta^{0,1} + \beta^{0,0} + \beta^{1,1}) = 0.\]
It follows from the proof of Lemma \ref{lem:betadef}, applied to each individual deformation $\Algd^i$, that
\[ (\delta'' + \delta''')(\beta^{0,2} + \beta^{0,1} + \beta^{0,0}) = 0.\]
So it remains to check that
\[ \delta'(\beta^{0,2} + \beta^{0,1} + \beta^{0,0}) + (\delta'+\delta''+\delta''')(\beta^{1,1}) = 0.\]
Indeed, one easily verifies that 
\begin{eqnarray*}
\delta'\beta^{0,2} + \delta'' \beta^{1,1} &=& 0,\\
\delta' \beta^{0,1} + \delta''' \beta^{1,1} &=& 0,\\
\delta'\beta^{0,0} &=& 0,\\
\delta'\beta^{1,1} &=& 0.
\end{eqnarray*}

Therefore, $\beta$ defines a chain map. 
Because $\beta^{0,0} = e \otimes e$ is the identity, its composition with the augmentation of the 2-extension coincides with the augmentation of the generalized simplicial bar resolution. 
It is clear that the composition with the map to $V^* \otimes \algd$ is the deformation class $\beta_\Algd = \beta^{0,2} + \beta^{1,1}$, as required: this completes the proof.
\end{proof}

\subsection{Derivations on the base}
\label{subsec:baseder}

Let $\BbK_\varepsilon := \BbK[\varepsilon]/\varepsilon^2$. 
Let $\xi \in \deriv_\Bbbk \BbK$ be a $\Bbbk$-relative derivation of $\BbK$. 
There is a corresponding map of $\BbK$-algebras
\begin{eqnarray*} 
\BbK & \to & \BbK_\varepsilon, \\
k & \mapsto & k + \varepsilon \cdot \xi(k).
\end{eqnarray*}
Hence, to any $\BbK$-algebra $\alg$ we can associate a deformation over $\BbK_\varepsilon$, \[ \Alg_\xi :=\alg \otimes_\BbK \BbK_\varepsilon,\]
where $\BbK_\varepsilon$ is regarded as a $\BbK$-algebra via the above map. 
We denote the corresponding deformation class by
\[\deform_\alg(\xi) := \deform_\alg(\Alg_\xi) \in \ext^2_{\alg\bimod}(\alg,\alg).\]

If we choose a $\BbK$-basis for $\alg$, we obtain a natural splitting for $\Alg_\xi$: namely, 
\[ s(a) := a\otimes 1 - \xi(a) \otimes \varepsilon.\]
Here, `$\xi(a)$' denotes the map which applies $\xi$ to the coefficients of $a$ with respect to the chosen $\BbK$-basis. 
With respect to this splitting, the deformation cocycle is
\[ \beta_{\Alg_\xi} = \xi(m_\alg),\]
i.e., the matrix with respect to the chosen $\BbK$-basis is obtained by applying $\xi$ to the matrix of the multiplication map $m_\alg$. 

The same construction applies to diagrams of $\BbK$-algebras: given a diagram $\algd$ and a derivation $\xi \in \deriv_\Bbbk \BbK$, we obtain a deformation of $\algd$ over $\BbK_\varepsilon$, namely
\[ \Algd_\xi := \algd \otimes_\BbK \BbK_\varepsilon,\]
and we denote the associated deformation class by
\[ \deform_\algd(\xi):=\deform_\algd(\Algd_\xi) \in \ext^2_{\algd\bimod}(\algd,\algd).\]
Now suppose we choose $\BbK$-bases for each $\algd^i$, and form the associated splittings $s^i$ as above: then we can write the deformation cocycle $\beta_{\Algd_\xi}$ explicitly. 
It has a component $\beta^{(i)}_{\Algd_\xi}$ for each $0$-simplex, which is equal to
\[ \xi(m_{\algd^i}) \co \algd^i \otimes \algd^i \to \algd^i,\]
i.e., the result of applying $\xi$ to the matrix of the multiplication map with respect to the chosen $\BbK$-basis; it also has a component $\beta^{(v)}_{\Algd_\xi}$ for each $1$-simplex, which is equal to
\[ \xi(\varphi^v) \co \algd^{cv} \to \algd^{dv},\]
i.e., the result of applying $\xi$ to the matrix of the restricting map $\varphi^v$, with respect to the chosen $\BbK$-bases on the domain and codomain.

The fact that $\beta_{\Algd_\xi}$ is a cocycle, and that it represents $\deform_\algd(\xi)$, follow from the results of the previous section. 
In this case, the proofs amount to nothing more than applying the product rule to the equations $m_{\algd^i}(m_{\algd^i}(\cdot,\cdot),\cdot) = m_{\algd^i}(\cdot,m_{\algd^i}(\cdot,\cdot))$ (associativity of $\algd^i$), $m_{\algd^{dv}}(\varphi^v(\cdot),\varphi^v(\cdot)) = \varphi^v(m_{\algd^{dv}}(\cdot,\cdot))$ ($\varphi^v$ is an algebra homomorphism), and $\varphi^u \varphi^v = \varphi^{uv}$ ($\algd$ is a functor).

\subsection{The diagram algebra}

Given a diagram of $\BbK$-algebras $\algd$, Gerstenhaber and Schack define the \emph{diagram algebra} $\algd!$, which is an ordinary $\BbK$-algebra; and they prove the \emph{special cohomology comparison theorem} \cite{Gerstenhaber1983,Gerstenhaber1988}, which implies that there is an isomorphism
\begin{equation}
\label{eqn:scct}
\ext^\bullet_{\algd\bimod}(\algd,\algd) \cong \ext^\bullet_{\algd!\bimod}(\algd!,\algd!).
\end{equation}
Furthermore, any deformation $\Algd$ of $\algd$ over $\BbK_\varepsilon$ induces a deformation of $\algd!$ over $\BbK_\varepsilon$, namely $\Algd!$; and the isomorphism \eqref{eqn:scct} takes $\deform_\algd(\Algd)$ to $\deform_{\algd!}(\Algd!)$, as one easily shows from the explicit cochain-level formula for the isomorphism \eqref{eqn:scct} derived in \cite[\S 17]{Gerstenhaber1983}. 

As a particular case of this, if $\xi \in \deriv_\Bbbk \BbK$, then the map \eqref{eqn:scct} sends $\deform_\algd(\xi)$ to $\deform_{\algd!}(\xi)$.

\subsection{Deformation classes of categories}

Let $\EuA$ be a $\BbK$-linear $A_\infty$ category, and $\xi \in \deriv_\Bbbk \BbK$. 
We have an associated deformation class $\deform_\EuA(\xi):=\KS_{cat}(\xi) \in \HH^2(\EuA)$. 
It can be defined by giving an explicit cochain-level representative: if we choose a $\BbK$-basis for each morphism space in $\EuA$, and write the matrices of the $A_\infty$ structure maps $\mu^*$ with respect to those bases, then $\deform_\EuA(\xi)$ is represented on the cochain level by $v(\mu^*)$. 
If $\EuB \subset \EuA$ is a full $A_\infty$ subcategory, there is an obvious restriction map
\[ CC^\bullet(\EuA) \to CC^\bullet(\EuB),\]
and it is obvious that this map takes $\deform_\EuA(\xi)$ to $\deform_\EuB(\xi)$.  
In fact one can check that the deformation class is a Morita invariant (see \cite[\S 4.4]{Sheridan2015a}). 

In particular, one can consider the special case of an ordinary $\BbK$-linear category $\EuA$, i.e., one for which the $A_\infty$ structure maps $\mu^s$ vanish for $s \neq 2$. 
Lowen and Van den Bergh \cite{Lowen2005} define the Hochschild cohomology of an abelian category $\EuC$ to be
\[ \pmb{\mathsf{H}}^\bullet_{\mathrm{ab}}(\EuC) := \HH^\bullet(\mathsf{Inj Ind}(\EuC));\]
we make the obvious definition
\[ \deform^{\mathrm{ab}}_{\EuC}(\xi) := \deform_{\mathsf{InjInd}(\EuC)}(\xi).\]
They prove \cite[Theorem 7.2.2]{Lowen2005} that for any diagram $\algd$, there is an isomorphism
\[ \pmb{\mathsf{H}}^\bullet_{\mathrm{ab}}(\algd\modd) \cong \ext^\bullet_{\algd!\bimod}(\algd!,\algd!),\]
which arises from the restriction map to the single-object subcategory of $\algd\modd$ consisting of a projective generator whose endomorphism algebra is $\algd!$. 
It follows that this isomorphism sends $\deform^{\mathrm{ab}}_{\algd\modd}(\xi)$ to $\deform_{\algd!}(\xi)$.

Now let $Y$ be a quasi-projective scheme over $\BbK$. 
Let $Open(Y)$ denote the poset of open affine subsets of $Y$.  
Let $Y = \cup_{i=1}^n A_i$ be a finite open affine covering of $Y$, and let $\cB \subset Open(Y)$ be the sub-poset consisting of $A_J := \cap_{i \in J} A_i$ for $\emptyset \neq J \subset \{1,\ldots,n\}$. 
We obtain a diagram of $\BbK$-algebras $\algd$ over $\cB$, namely $\algd^J := \mathcal{O}_Y(A_J)$, with the obvious restriction maps. 

Lowen and Van den Bergh prove \cite[Corollary 7.7.3]{Lowen2005} that there is an isomorphism
\[ \pmb{\mathsf{H}}^\bullet_{\mathrm{ab}}(Coh(Y)) \cong \pmb{\mathsf{H}}^\bullet_{\mathrm{ab}}(\algd\modd).\]
The isomorphism respects deformation classes. 

Finally, they consider a certain $\mathsf{dg}$ enhancement of $D^bCoh(Y)$, which they denote $^eD^b(Coh(Y))$, and we will denote $\dbdg{Y}$ (recalling that the $\mathsf{dg}$ enhancement is unique up to quasi-equivalence). 
They prove \cite[Theorem 6.1]{Lowen2005} that there is an isomorphism
\[ \pmb{\mathsf{H}}^\bullet_{\mathrm{ab}}(Coh(Y)) \cong \HH^\bullet(\dbdg{Y}).\]
The isomorphism respects deformation classes. 

\subsection{Deformations of schemes}
\label{subsec:schdef}

As in the previous section, let $Y$ be a quasi-projective scheme over $\BbK$.
Swan \cite{Swan1996} defines the Hochschild cohomology of $Y$ to be
\[ \HH^\bullet(Y) := \ext_{Y \times Y}(\Delta_* \mathcal{O}_Y,\Delta_*\mathcal{O}_Y),\]
where $\Delta \co Y \hookrightarrow Y \times Y$ is the inclusion of the diagonal. 

The construction of the algebraic Atiyah class globalizes. 
Namely, we have the short exact sequence
\begin{equation}
\label{eqn:atsch}
 0 \to \Delta_* \Omega^1_Y \to \mathcal{O}_{\Delta^{(2)}Y} \to \Delta_* \mathcal{O}_Y \to 0,
\end{equation}
where $\mathcal{O}_{\Delta^{(2)}Y}$ is the second infinitesimal neighbourhood of the diagonal. 
This short exact sequence gives rise to a morphism $\Delta_* \mathcal{O}_Y \to \Delta_* \Omega^1_Y[1]$ in $D^bCoh(Y \times Y)$, whose associated class
\[ \At_Y \in \ext^1_{Y\times Y} (\Delta_*\mathcal{O}_Y, \Delta_* \Omega^1_Y )\] 
is the standard geometric Atiyah class. 

The construction of the algebraic Kodaira--Spencer class also globalizes. 
Let $V$ be a $\BbK$-vector space, and $\BbK_\varepsilon := K[V]/V^2$ as before. Let $\mathcal{Y}$ be a deformation of $Y$ over $\BbK_\varepsilon$, i.e., a scheme $\mathcal{Y}$ over $\spec \BbK_\varepsilon$, equipped with an isomorphism $\mathcal{Y} \times_{\spec \BbK_\varepsilon} \spec \BbK \cong Y$. 

Let $i \co Y \to \mathcal{Y}$ denote the inclusion of the central fibre of the deformation.  The conormal short exact sequence
\[ 0 \to V^*\otimes \mathcal{O}_Y \to i^* \Omega^1_{\mathcal{Y}} \to \Omega^1_Y \to 0\]
gives rise to a morphism $\Omega^1_Y \to V^* \otimes \mathcal{O}_Y[1]$ in $D^bCoh(Y)$, whose associated class
\begin{eqnarray*}
\KS_\mathcal{Y} & \in & \ext^1_Y(\Omega^1_Y,V^* \otimes \mathcal{O}_Y) \\
& \cong  & \Hom(V,H^1(Y,\mathcal{T}Y))
\end{eqnarray*}
is called the \emph{Kodaira--Spencer map}. 

The pushforward short exact sequence
\[ 0 \to V^* \otimes \Delta_* \mathcal{O}_Y \to \Delta_* i^* \Omega^1_\mathcal{Y} \to \Delta_* \Omega^1_Y \to 0\]
gives rise to a morphism $\Delta_* \Omega^1_Y \to V^* \otimes \Delta_* \mathcal{O}_Y$ in $D^bCoh(Y \times Y)$. 
We define the \emph{deformation class} of $\mathcal{Y}$, $\deform_Y(\mathcal{Y}) \in V^* \otimes \HH^2(Y)$, to be the composition of these two maps in $D^bCoh(Y \times Y)$: 
\[ \deform_Y(\mathcal{Y}) := \KS_\mathcal{Y} \circ \At_Y \in \ext^2_{Y \times Y}(\Delta_* \mathcal{O}_Y,V^* \otimes \Delta_*\mathcal{O}_Y).\] 
It follows from a result of C\u{a}ld\u{a}raru \cite[Proposition 4.4]{Caldararu2005} that for any class $\alpha \in H^1(Y,\mathcal{T}Y) \cong \ext^1_Y(\Omega^1_Y,\mathcal{O}_Y)$, 
\[ \HKR(\alpha) = \Delta_*\alpha \circ \At_Y,\]
so in particular,
\begin{equation}
\label{eqn:caleq}
\HKR \circ \KS_\mathcal{Y} = \deform_Y(\mathcal{Y}).
\end{equation}

The deformation class is represented by the 2-extension
\begin{equation}
\label{eqn:2extsch}
 0 \to V^* \otimes \mathcal{O}_Y \to \Delta_* i^* \Omega^1_\mathcal{Y} \to \mathcal{O}_{\Delta^{(2)}Y} \to \Delta_* \mathcal{O}_Y \to 0.
\end{equation}

Now, let $Y = \cup_{i=1}^nA_i$ a finite open affine covering as in the previous section, and $\algd$ the associated diagram of $\BbK$-algebras.
There is an obvious exact functor $Coh(Y \times Y) \to \algd\bimod$, sending $\mathcal{F} \mapsto \mathbb{F}$, where $\mathbb{F}^J := \mathcal{F}(A_J \times A_J)$. 
It obviously sends $\Delta_* \mathcal{O}_Y \mapsto \algd$.
Swan proves \cite[Theorem 3.1]{Swan1996} that this functor gives rise to an isomorphism
\begin{equation}
\label{eqn:swaniso}
 \ext^\bullet_{Y \times Y}(\Delta_* \mathcal{O}_Y,\Delta_* \mathcal{O}_Y) \cong \ext^\bullet_{\algd\bimod}(\algd,\algd).
\end{equation}
This functor obviously sends the short exact sequence \eqref{eqn:2extsch} to the short exact sequence \eqref{eqn:defdcomm} defining the deformation class in the commutative case. 
Hence, the isomorphism $V^* \otimes \eqref{eqn:swaniso}$ respects deformation classes. 

As a particular case of the above, let $\xi \in \deriv_\Bbbk \BbK$: then we obtain a deformation $\mathcal{Y}_\xi := Y \times_\BbK \BbK_\varepsilon$, where $\BbK_\varepsilon$ is regarded as a $\BbK$-algebra via a map determined by $\xi$ as in \S \ref{subsec:baseder}. 
We define the \emph{classical Kodaira--Spencer map} (which appears in the statement of Proposition \ref{prop:KSHKR}) to be
\begin{eqnarray*}
 \KS_{class}\co  \deriv_\Bbbk \BbK &\to& H^1(Y,\mathcal{T}Y) \\
\KS_{class}(\xi) &:=& \KS_{\mathcal{Y}_\xi}(1)
\end{eqnarray*}
(here, `$1$' is regarded as an element of $V \cong \BbK$). 
We define the deformation class $\deform_Y(\xi) := \deform_Y(\mathcal{Y}_\xi)$.
It follows from \eqref{eqn:caleq} that
\[ \HKR(\KS_{class}(\xi)) = \deform_Y(\xi).\]

This completes the proof of Proposition \ref{prop:KSHKR}: we have explicitly identified the isomorphisms
\[ HT^\bullet(Y) \cong \HH^\bullet(Y) \cong \ldots \cong \HH^\bullet(\dbdg{Y}),\]
shown that the first one takes the Kodaira--Spencer class $\KS_{class}(\xi) \in HT^2(Y)$ to the deformation class $\deform_Y(\xi) \in \HH^2(Y)$, and shown that all subsequent isomorphisms respect deformation classes, up until $\deform_{\dbdg{Y}}(\xi) := \KS_{cat}(\xi)$.

\bibliographystyle{amsalpha}
\bibliography{library}

\end{document}